\documentclass[letter,11pt,reqno]{amsart}

\usepackage[utf8]{inputenc}

\usepackage{etex}

\usepackage{xcolor}

\definecolor{verydarkblue}{rgb}{0,0,0.5}

\usepackage[
    breaklinks,
    colorlinks,
    citecolor=verydarkblue,
    linkcolor=verydarkblue,
    urlcolor=verydarkblue,
    pagebackref=true,
    hyperindex
]{hyperref}

\backrefenglish

\usepackage{fancyhdr}

\usepackage[
    hscale=0.70,
    vscale=0.775,
    headheight=13pt,
    centering,
]{geometry}

\usepackage{amsmath}
\usepackage{amsthm}
\usepackage{amssymb}
\usepackage{bm}
\usepackage{mathtools}
\usepackage{stmaryrd}
\usepackage{mathdots}
\usepackage{framed}
\usepackage[capitalize]{cleveref}
\usepackage{array}
\usepackage[alphabetic]{amsrefs}
\usepackage[all,cmtip]{xy}

\theoremstyle{plain}

\crefname{introtheorem}{Theorem}{Theorems}

\newtheorem{theorem}{Theorem}[section]
\newtheorem{proposition}[theorem]{Proposition}
\newtheorem{lemma}[theorem]{Lemma}

\newtheorem{definition-theorem}[theorem]{Definition-Theorem}

\theoremstyle{definition}

\theoremstyle{remark}

\newtheorem{remark}[theorem]{Remark}
\newtheorem{example}[theorem]{Example}

\numberwithin{figure}{section}

\numberwithin{equation}{section}

\IfFileExists{./article-style.tex}{\input{article-style.tex}}{}

\usepackage{marginnote}

\def\N{{\mathbb N}}
\def\Z{{\mathbb Z}}
\def\Q{{\mathbb Q}}
\def\R{{\mathbb R}}

\def\S{{\mathbb S}}

\def\S{{\mathbb S}}

\def\A{{\mathbb A}}
\def\P{{\mathbb P}}

\def\cH{\mathcal{H}}

\def\I{\mathcal{I}}
\def\J{\mathcal{J}}
\def\O{\mathcal{O}}

\def\fa{\mathfrak{a}}

\def\fm{\mathfrak{m}}
\def\fn{\mathfrak{n}}

\def\a{\alpha}

\def\g{\gamma}

\def\f{\phi}
\def\ff{\psi}

\def\p{\pi}

\def\t{\tau}
\def\x{\xi}

\def\D{\Delta}
\def\G{\Gamma}

\def\S{\Sigma}
\def\Om{\Omega}

\def\.{\cdot}
\def\^{\widehat}
\def\~{\widetilde}
\def\o{\circ}

\def\rat{\dashrightarrow}
\def\surj{\twoheadrightarrow}

\def\lru{\lceil}
\def\rru{\rceil}
\def\lrd{\lfloor}
\def\rrd{\rfloor}
\def\({\left(}
\def\){\right)}

\makeatletter
\newcommand*{\da@rightarrow}{\mathchar"0\hexnumber@\symAMSa 4B }
\newcommand*{\da@leftarrow}{\mathchar"0\hexnumber@\symAMSa 4C }
\newcommand*{\xdashrightarrow}[2][]{%
  \mathrel{%
    \mathpalette{\da@xarrow{#1}{#2}{}\da@rightarrow{\,}{}}{}%
  }%
}
\newcommand{\xdashleftarrow}[2][]{%
  \mathrel{%
    \mathpalette{\da@xarrow{#1}{#2}\da@leftarrow{}{}{\,}}{}%
  }%
}
\newcommand*{\da@xarrow}[7]{%
  \sbox0{$\ifx#7\scriptstyle\scriptscriptstyle\else\scriptstyle\fi#5#1#6\m@th$}%
  \sbox2{$\ifx#7\scriptstyle\scriptscriptstyle\else\scriptstyle\fi#5#2#6\m@th$}%
  \sbox4{$#7\dabar@\m@th$}%
  \dimen@=\wd0 %
  \ifdim\wd2 >\dimen@
    \dimen@=\wd2 %
  \fi
  \count@=2 %
  \def\da@bars{\dabar@\dabar@}%
  \@whiledim\count@\wd4<\dimen@\do{%
    \advance\count@\@ne
    \expandafter\def\expandafter\da@bars\expandafter{%
      \da@bars
      \dabar@ 
    }%
  }%
  \mathrel{#3}%
  \mathrel{%
    \mathop{\da@bars}\limits
    \ifx\\#1\\%
    \else
      _{\copy0}%
    \fi
    \ifx\\#2\\%
    \else
      ^{\copy2}%
    \fi
  }%
  \mathrel{#4}%
}
\makeatother

\newcommand{\rd}[1]{\lrd{#1}\rrd}
\newcommand{\ru}[1]{\lru{#1}\rru}

\renewcommand{\and}{ \ \ \text{ and } \ \ }

\def\reg{\mathrm{reg}}

\def\Jac{\mathrm{Jac}}

\DeclareMathOperator{\codim} {codim}

\DeclareMathOperator{\Spec} {Spec}

\DeclareMathOperator{\ord} {ord}

\DeclareMathOperator{\vol} {vol}
\DeclareMathOperator{\val} {val}

\DeclareMathOperator{\lct} {lct}

\DeclareMathOperator{\Fitt} {Fitt}

\DeclareMathOperator{\Bs} {Bs}

\DeclareMathOperator{\mld} {mld}
\DeclareMathOperator{\ct} {ct}

\makeatletter
\@namedef{subjclassname@2020}{%
  \textup{2020} Mathematics Subject Classification}
\makeatother

\begin{document}

\title[Fano hypersurfaces with ordinary double points]
{Birational rigidity and K-stability of Fano hypersurfaces with ordinary double points}

\author{Tommaso de Fernex}
\address{Department of Mathematics, University of Utah, Salt Lake City, UT 48112, USA}
\email{{\tt defernex@math.utah.edu}}


\thanks{%
The research of the first author was partially supported by NSF Grant DMS-2001254.
}

\begin{abstract}
Extending previous results, we prove that for $n \ge 5$ all
hypersurfaces of degree $n+1$ in $\P^{n+1}$ with isolated
ordinary double points are birational superrigid and K-stable, hence admit a weak K\"ahler--Einstein metric.
\end{abstract}

\maketitle

\section{Introduction}

This work continues prior research on birational superrigidity of Fano hypersurfaces in projective spaces, 
a property of certain Mori fiber spaces that was first identified for smooth quartic threefolds
as a byproduct of Iskovskikh and Manin's proof of nonrationality. 

The starting point of the paper is the following theorem.

\begin{theorem}
\label{t:smooth}
For $n \ge 3$, every smooth hypersurface of degree $n+1$ in $\P^{n+1}$ is birationally superrigid, 
hence nonrational. 
\end{theorem}

This theorem has a long history that covers over a century, starting 
with the work of Fano on quartic threefolds \cite{Fan07,Fan15} which was 
later revisited and completed by Iskovskikh and Manin \cite{IM71}. 
The higher dimensional case was promoted in the work of Pukhlikov \cite{Puk87,Puk98},
where significant advances were made, and a complete proof 
was eventually given in \cite{dF16}; the proof was later simplified in \cite{Zhu20}.
Other people have contributed with key ideas along the way;
we recall here Segre and Manin's works of cubic surfaces over nonclosed fields \cite{Seg51,Man66},  
Corti's new proof of the three-dimensional case where inversion of adjunction is used for the first time
to connect the method of maximal singularities to log canonical thresholds \cite{Cor00}, 
and the author's work with Ein and Musta\c t\u a on log canonical thresholds \cite{dFEM03,dFEM04}. 
A survey of the history of \cref{t:smooth} can be found in \cite{Kol19}. 
We will briefly review the proof of \cref{t:smooth} following \cite{Zhu20,Kol19}
before turning our attention to the case of singular hypersurfaces. 

Birational superrigidity is sensitive to singularities. 
Singular quartic threefolds in $\P^4$ are never birationally superrigid, 
though under some reasonable assumptions those with ordinary double points are 
birational rigid, which is a weaker condition compared to superrigidity \cite{Puk88,CM04,Mel04}.
In higher dimensions, Pukhlikov extended his proof of birational superrigidity of general smooth Fano hypersurfaces
to general hypersurfaces with isolated ordinary multiple points \cite{Puk02b}, and
with Eckl, they prove a similar result for general hypersurfaces 
with quadratic (non-necessarily isolated) singularities of rank $\ge 5$ \cite{EP14}. 
Cheltsov looked at sextic fivefolds in $\P^6$ with isolated 
ordinary double points, proving that they are always birationally superrigid provided
they do not contain any plane \cite{Che07}. 

More recently, birational superrigidity of singular hypersurfaces was established in \cite{dF17}
under rather broad assumptions on the singularities, where a certain measure of the singularity
is bounded in terms of the dimension. 
An asymptotically stronger bound was later given under more restrictive conditions on the 
types of singularities in \cite{LZ21}. 
While these results are asymptotically strong as the dimension grows, 
they leave a gap open in low dimensions. 

Even restricting to the simplest case of isolated ordinary 
double points (namely, points whose tangent cones are cones
over smooth quadrics), the main theorem of \cite{dF17} only covers the cases of dimension $n \ge 19$, 
leaving Cheltsov's theorem on sextic fivefolds not containing planes 
and Pukhlikov's result for general hypersurfaces
with isolated ordinary double points the only known results in the remaining cases $4 \le n \le 18$. 

The purpose of this paper is to try to fill this gap. 

\begin{theorem}
\label{t:ODP}
For $n \ge 5$, every hypersurface of degree $n+1$ in $\P^{n+1}$
with only isolated ordinary double points as singularities is birationally superrigid, hence nonrational. 
\end{theorem}

To prove this result, we provide a new asymptotic argument which 
allows us to cover all cases $n \ge 8$, thus reducing the gap to a few lower dimensional 
cases. The remaining cases are treated using an extension to singular hypersurfaces
of the projection method of \cite{dFEM03}, which is made possible using a result from \cite{dFM15}
on Mather log discrepancies.
The proof does not extend to the case $n=4$, 
and the question whether all quintic hypersurfaces in $\P^5$
with isolated ordinary double points are birationally superrigid remains open.

Birational superrigidity is a property closely related to K-stability. 
A first indication of this was discovered in \cite{OO13}, and a more precise connection was
established in \cite{SZ19,Zhu20}. 
It was observed in particular in \cite{SZ19} that \cref{t:smooth}, 
combined with Cheltsov's bound on the $\a$-invariant \cite{Che01}, implies that all 
smooth hypersurfaces of degree $n+1$ in $\P^{n+1}$ are K-stable, hence admit
a K\"ahler--Einstein metric. 
Note that the same conclusion follows from by combining Cheltsov's bound to
\cite{Fuj19b}, which extends Tian's criterion for K-stability of
Fano manifolds \cite{Tia87} to the boundary case.
 
Under suitable generality conditions on the defining equation 
which are met by general hypersurfaces of degree $n+1$ in $\P^{n+1}$ with 
quadratic singularities of rank $\ge 8$, 
Pukhlikov established stronger bounds on $\a$-invariant and canonical thresholds, see
\cite{Puk05,Puk15,Puk21}. These bounds immediately imply K-stability 
by Tian's criterion and its extension to singular Fano varieties due to \cite{OS12};
note, however, that these stronger bounds do not holds without some generality conditions.
See also \cite{LZ19} regarding the sharpness of Tian's criterion and its aforementioned extensions. 

Here we apply \cite{SZ19}, along with \cite{LXZ21}, to deduce the following result from \cref{t:ODP}.

\begin{theorem}
\label{t:ODP-K-stab}
For $n \ge 5$, every hypersurface of degree $n+1$ in $\P^{n+1}$ 
with only isolated ordinary double points as singularities is K-stable and
hence admits a weak K\"ahler--Einstein metric. 
\end{theorem}

\subsubsection*{Acknowledgements} 
We would like to thank Aleksandr Pukhlikov for bringing to our attention several relevant references,
Kento Fujita and Yuga Tsubouchi for pointing out an error in the proof of an earlier version of 
\cref{l:sharp-bound-length}, 
and the referee for valuable comments and for suggesting 
an extension of \cref{t:ODP} to quadratic singularities of rank $\ge 2$ (see \cref{r:rank>1}).

\section{Preliminaries on singularities}

We work over an algebraically closed field of characteristic zero.

\subsection{Multiplicities}

We start this section by recalling some basic properties of multiplicities. We follow \cite{Ful98}.
In particular, for us the \emph{multiplicity} $e_p(X)$ of 
a pure dimensional scheme $X$ 
at a closed point $p$ is defined to be the Samuel multiplicity $e(\fm_p)$ 
of the maximal ideal $\fm_p$ of the local ring $\O_{X,p}$.

\begin{example}
\label{eg:classical-mult}
If $X \subset \A^n$ is a pure dimensional Cohen--Macaulay scheme and $p \in X$ is a closed point, 
then $e_p(X) = i(p, X\.L; \A^n)$, the intersection multiplicity 
of $X$ with a general linear subspace $L \subset \A^n$ of codimension 
equal to the dimension of $X$. This follows from \cite[Example~4.3.5(c) and Proposition~7.1]{Ful98}. 
It relates the definition of multiplicity adopted here with more classical definitions. 
The equality does not hold in general if $X$ is not Cohen--Macaulay. 
\end{example}

\begin{example}
If $D$ is an effective Cartier divisor on a variety $X$ 
and $p \in X$ is a regular point, then $e_p(D)$ is simply the
order of vanishing of $D$ at $p$, namely, the multiplicity of a
generator of the ideal of $D$ in the local ring $\O_{X,p}$ \cite[Example~4.3.9]{Ful98}.
This is no longer true, however, if $p$ is a singular point of $X$ 
(cf.\ \cref{r:order-of-vanishing} below).
\end{example}

More generally, 
for any closed subscheme $Z$ of a pure dimensional scheme $X$ and any irreducible component $T$ of $Z$, 
the {\it multiplicity of $X$ along $Z$ at $T$}, denoted by $e_Z(X)_T$, 
is defined to be the Samuel multiplicity of the ideal of $Z$ 
in the local ring $\O_{X,T}$; equivalently, $e_Z(X)_T$ is the coefficient of 
$[T]$ in the Segre class of the normal cone to $Z$ in $X$ \cite[Example~4.3.4]{Ful98}. 
If $Z = T$, then we just write $e_T(X)$. 

\begin{example}
\label{eg:e_pZ=e_ZX_p}
Whenever $T$ is an irreducible component of a scheme $Z$ we have $e_T(Z) = l(\O_{Z,T})$ by definition, 
and if $Z$ is regularly embedded in a pure dimensional Cohen--Macaulay scheme $X$ then 
$e_T(Z) = e_Z(X)_T$ \cite[Example~4.3.5(c)]{Ful98}. 
\end{example}

\begin{example}
\label{eg:ci-mult}
If $Z = D_1 \cap \dots \cap D_r \subset X$ is the complete intersection of $r$ divisors
on variety $X$ and $T$ is an irreducible component of $Z$, then 
$e_Z(X)_T = i(T,D_1\. \dots \. D_r;X)$,
where the last term denotes the intersection multiplicity of the divisors $D_i$
at the generic point of $T$ \cite[Example~7.1.10.(a)]{Ful98}.
\end{example}

Multiplicities are semicontinuous, in the following sense. 
If $X$ is a pure-dimensional scheme, then
for every irreducible closed set $T \subset X$
there is a nonempty open set $T^\o \subset T$ such that
$e_p(X) \ge e_T(X)$ for every point $p \in X$, and equality holds if $p \in T^\o$.
This is proved in \cite[Theorem~(4)]{Ben70}; we will tacitly use this fact without further mention.

\begin{remark}
\label{r:order-of-vanishing}
Differently from the multiplicity, the order of vanishing
of a Cartier divisor $D$ on a variety $X$ is not always a semicontinuous 
function of the point of the variety. 
For example, take $X \subset \A^3$ to be a cone over a smooth conic
and $D = 2L$ where $L$ is a line through the vertex of the cone. 
The order of vanishing $D$ at any point of $L$ is 2
except at the vertex, where the order of vanishing drops to 1. 
\end{remark}

We list a few properties of multiplicities that will be used later in the paper.
The first property, which can be interpreted as a Bertini-type property, is proved in
\cite[Proposition~4.5]{dFEM03} and \cite[Proposition~8.5]{dF13}.

\begin{proposition}
\label{p:Bertini}
Let $Z$ be a pure-dimensional closed Cohen-Macaulay subscheme of $\P^m$ of positive dimension
and $\cH \subset (\P^m)^\vee$ a hyperplane in the dual space. 
Then for a general $H \in \cH$ we have $e_p(Z \cap H) = e_p(Z)$ for every $p \in Z \cap H$. 
\end{proposition}

The proof of the next Bezout inequality combines various properties established in \cite{Ful98}
and is included for the convenience of the reader. 

\begin{proposition}
\label{p:Bezout1}
Let $Z = X \cap H_1 \cap \dots \cap H_r$ be the complete intersection of a variety $X \subset \P^n$ with 
$r$ hypersurfaces $H_i \subset \P^n$.
Then for every irreducible component $T$ of $Z$ we have
\[
e_Z(X)_T \le \deg X \.\prod \deg H_i.
\]
\end{proposition}

\begin{proof}
As we recalled in \cref{eg:ci-mult}, letting $D_i = H_i \cap X$ we have $e_Z(X)_T = i(T,D_1\.\dots\.D_r;X)$.
By the definition of intersection product, this number is bounded above
by the intersection product $(D_1\.\dots\.D_r)_X$, which in turns is equal to
$(H_1\.\dots\.H_r\.[X])_{\P^n}$, see \cite[Example~2.4.3]{Ful98}.
Therefore the proposition follows by Bezout theorem \cite[Propostion~8.4]{Ful98}.
\end{proof}

Finally, we recall the following multiplicity bound on cycles 
on projective hypersurfaces due to \cite[Proposition~5]{Puk02}
(cf.\ \cite[Remark~4.4]{dFEM03}).

\begin{proposition}
\label{p:Puk}
Let $X \subset \P^n$ be a smooth hypersurface 
and $Z \subset X$ a pure-dimensional closed scheme with no embedded points. 
Let $k = \codim(Z,X)$, and assume that the cycle $[Z]$ is numerically equivalent
to $m\.c_1(\O_X(1))^k$. Then $\dim\{ x \in X \mid e_x(Z) > m\} < k$.
\end{proposition}

\begin{remark}
\label{r:Puk}
The example discussed in \cref{r:order-of-vanishing} shows that \cref{p:Puk}
does not hold in general if the hypersurface is singular. 
However, one can always cut down away from the singularities and use 
\cref{p:Bertini} to generalize the proposition to singular hypersurfaces. 
In particular, if $X \subset \P^n$ is a hypersurface with isolated singularities
then $\dim\{ x \in X \mid e_x(Z) > m\} \le k$, 
and we have $e_T(Z) \le m$ for every closed subvariety $T \subset X$ of dimension $k$
that is disjointed from the singular locus of $X$.
\end{remark}

\subsection{Singularities of pairs}

Let $X$ be a positive dimensional normal variety 
and assume that the canonical divisor $K_X$ of $X$ is $\Q$-Cartier. 

We consider \emph{pairs} of the form $(X,Z)$ where $Z = \sum c_i Z_i$ is a \emph{$\Q$-subscheme}, 
i.e., a proper linear combination of proper closed subschemes $Z_i \subset X$ with coeffients $c_i$
in $\Q$. 
We say that the $\Q$-subscheme $Z$ (or the pair $(X,Z)$) is \emph{effective} if $c_i \ge 0$ for all $i$. 
If $\cH$ is a linear system on $X$, then we denote by $(X,c\cH)$ the pair $(X,c\Bs(\cH))$
where $\Bs(\cH)$ is the base scheme of $\cH$.
We will also consider pairs of the form $(X,\fa^c)$ where $\fa \subset \O_X$ is a nonzero 
coherent ideal sheaf and $c \in\Q$, and $(X,\D)$ where $\D$ is a $\Q$-Cartier $\Q$-divisor; 
both versions can be view in an obvious way as special cases of the above definition of pair. 

We will refer to any prime divisor $E$ on a resolution of $X$ as a \emph{divisor over} $X$, 
and denote by $c_X(E) \subset X$ the closure of the center of the corresponding valuation $\val_E$.
With small abuse of terminology, we call $c_X(E)$ the \emph{center} of $E$ in $X$, 
and identify two divisors over $X$ whenever they define the same valuation.

Given a pair $(X,Z)$ and a divisor $E$ over $X$, we define the \emph{log discrepancy}
of $(X,Z)$ along $E$ to be 
\[
a_E(X,Z) := \ord_E(K_{Y/X}) + 1 - \val_E(Z),
\]
where $Y$ is a resolution of $X$ where $E$ lives, $K_{Y/X}$ is the relative canonical divisor
(a $\Q$-divisor), and $\val_E(Z) := \sum c_i \val_E(\I_{Z_i})$
where $\I_{Z_i} \subset \O_X$ is the ideal sheaf of $Z_i$. 
If $Z = 0$, then we simply write $a_E(X)$ (the same quantity is also denoted by $A_X(E)$ in the 
literature).
The \emph{minimal log discrepancy} $\mld_p(X,Z)$ of a pair $(X,Z)$ at a point $p \in X$
is the infimum of the log discrepancies $a_E(X,Z)$ of all divisors $E$ over $X$
with center $p$.

\begin{remark}
Once $E$ is fixed, if $\cH$ is a linear system and $D_i \in \cH$ are general members of it, 
then $a_E(X,\cH) = a_E(X,\bigcap D_i)$. Using resolution of singularities, one also sees
that once a point $p \in X$ is fixed, if $D_i \in \cH$ are general members 
then $\mld_p(X,\cH) = \mld_p(X,\bigcap D_i)$. 
\end{remark}

The next property is an immediate consequence of Shokurov--Koll\'ar's connectedness principle
\cite{Sho92,Kol92}. 

\begin{theorem}[Inversion of Adjunction]
\label{t:inv-of-adj}
Let $(X,Z)$ be an effective pair where $Z = \sum c_iZ_i$, let $H \subset X$ 
be a normal irreducible Cartier divisor
that is not contained in any $Z_i$, and set $Z|_H := \sum c_i (Z_i \cap H)$. 
Suppose $p \in H$ is a closed point where $\mld_p(X,Z + H) \le 0$. 
Then $\mld_p(H,Z|_H) \le 0$.
\end{theorem}

\begin{remark}
The converse of \cref{t:inv-of-adj} also holds, namely, we have 
$\mld_p(X,Z + H) \le 0$ whenever $\mld_p(H,Z|_H) \le 0$.
This direction is a simple consequence of the adjunction formula.
\end{remark}

Assuming $\dim X \ge 2$, a pair $(X,Z)$ is said to be \emph{terminal} (resp., \emph{canonical})
if $a_E(X,Z) > 1$ (resp., $\ge 1$) for all exceptional divisors $E$ over $X$. 
A pair $(X,Z)$ is said to be \emph{log terminal} (resp., \emph{log canonical})
if $a_E(X,Z) > 0$ (resp., $\ge 0$) for all divisors $E$ over $X$. 

Let $(X,Z)$ be an effective pair. 
If $\dim X \ge 2$ and $X$ is canonical, then we denote by $\ct(X,Z)$ the \emph{canonical threshold},  
namely, the supremum of the set of coefficients $c$ such that $(X,cZ)$ is canonical.
If $X$ is log canonical, then we denote by $\lct(X,Z)$ the \emph{log canonical threshold}, 
namely, the supremum such that $(X,cZ)$ is log canonical.
For a point $p \in X$, we denote by $\ct_p(X,Z)$ and $\lct_p(X,Z)$ the 
canonical and log canonical thresholds of 
the pair restricted to any sufficiently small open neighborhood of $p$. 

The next proposition relates thresholds to multiplicities. 
The first property stated in the proposition 
follows from inversion of adjunction (e.g., see \cite[Proposition~8.8]{dF13}).
The second is proved in \cite{dFEM04}. 

\begin{proposition}
\label{t:mult-thresholds}
Let $X$ be a smooth $n$-dimensional variety and $p \in X$ a closed point.  
\begin{enumerate}
\item
If $n \ge 2$, then for any effective divisor $D$ on $X$, we have $e_p(D)\ge 1/\ct_p(X,D)$. 
\item
For any zero-dimensional subscheme $Z$ of $X$, we have
$e_Z(X)_p\ge n^n/\lct_p(X,Z)^n$.
\end{enumerate}
\end{proposition}

The \emph{multiplier ideal} of an effective pair $(X,Z)$, where $Z = \sum c_i Z_i$, is defined by taking 
a log resolution $f \colon Y \to X$ of the pair and setting 
\[
\J(X,Z) := f_*(\ru{K_{Y/X} - f^{-1}(Z)})
\]
where $f^{-1}(Z)$ denotes the $\Q$-divisor $\sum c_i f^{-1}(Z_i)$.
It is a standard computation to check that the definition is independent of the choice of resolution. 
Multiplier ideals satisfy the following vanishing theorem due to Nadel (see \cite[Section~9.4]{Laz04}).

\begin{theorem}[Nadel Vanishing Theorem]
\label{t:Nadel}
Let $(X,Z)$ be an effective pair with $Z = \sum c_iZ_i$, 
let $A_i$ be Cartier divisors on $X$ such that $\O_X(A_i) \otimes \I_{Z_i}$ is globally generated
for every $i$, and let $A$ be a Cartier divisor such that
the $\Q$-divisor $A - K_X - \sum c_iA_i$ is nef and big. Then 
$H^j(X,\O_X(A) \otimes \J(X,Z)) = 0$ for all $j \ge 1$.
\end{theorem}

\subsection{Bounds on colength}

Here we review the main new ingredient of \cite{Zhu20} as revisited in \cite{Kol19}. 

\begin{proposition}
\label{p:colength-mult-ideal}
Let $X$ be a smooth variety of dimension $n \ge 2$,
and $Z$ an effective $\Q$-scheme on $X$ with $\lct(X,Z) \le 1$.
If $\S \subset X$ is the subscheme defined by the multiplier ideal $\J(X,2Z) \subset \O_X$,
then 
\[
l(\O_\S) \ge \tfrac 12 3^n.
\]
\end{proposition}

We briefly review the proof. We start with two lemmas. 

\begin{lemma}
\label{l:double-log-discr}
Let $(X,Z)$ be an effective pair.
Then for every divisor $E$ over $X$ we have
\[
a_E(X,\J(X,2Z)) < 2 a_E(X,Z).
\]
In particular, if $(X,Z)$ is not log terminal, then 
$(X,\J(X,2Z))$ is not log canonical.
\end{lemma}

\begin{proof}
Denote for short $\J =\J(X,2Z)$. 
We may assume that $E$ is a prime divisor on a log resolution $f \colon Y \to X$
of $(X,Z)$. Set $k_E = \ord_E(K_{Y/X})$, $z_E = \ord_E(Z)$, and $j_E = \ord_E(\J)$.
By definition, $\J = f_*\O_Y(- \rd{2f^{-1}(Z) - K_{Y/X}})$, hence
\[
j_E \ge \ord_E(\rd{2f^{-1}(Z) - K_{Y/X}}) = \rd{2z_E - k_E} > 2z_E - k_E - 1.
\]
This implies that 
\[
a_E(X,\J) = k_E + 1 - j_E < 2k_E + 2 - 2z_E = 2 a_E(X,Z),
\]
as claimed.
\end{proof}

\begin{lemma}
\label{l:sharp-bound-length}
Let $R = \O_{X,x}$ be the local ring of a variety $X$ at a regular point $x$. 
Assume that $\dim R = n \ge 2$, and let $\fa \subset R$ be an $\fm$-primary ideal with $\lct(\fa) < 1$.
Then 
\[
l(R/\fa) \ge \tfrac 12 3^n.
\]
\end{lemma}

\begin{proof}
A standard argument (see, e.g., the proof of \cite[Theorem~1.1]{dFEM04}) allows to reduce to the case where $\fa$
is a monomial ideal in a polynomial ring $k[x_1,\dots,x_n]$. 
The condition that $\lct(\fa) < 1$ implies that
there are positive numbers $a_i$ such that the tetrahedron 
\[
T = \big\{ u \in \R_{\ge 0}^n \mid \sum a_i u_i \le \sum a_i \big\}
\]
is disjoint from the Newton polyhedron $P(\fa)$. Hence $l(R/\fa) \ge | T \cap \Z^n |$.
The upperbound $\tfrac 12 3^n$ accounts for the fact that 
for any $(u_1,\dots,u_n) \in \Z^n$ with $u_i \in \{0,1,2\}$ for all $i$, either 
$(u_1,\dots,u_n)$ or $(2-u_1,\dots,2-u_n)$ belongs to $T$. 
\end{proof}

\begin{proof}[Proof of \cref{p:colength-mult-ideal}]
We may assume that $\S$ is zero dimensional as otherwise $l(\O_\S) = \infty$. 
Since $(X,Z)$ is not log terminal, \cref{l:double-log-discr} implies that 
$(X,\S)$ is not log canonical, that is, $\lct(X,\S) < 1$.
Then there is a connected component $\S'$ of $\S$ such that $\lct(X,\S') < 1$, 
hence $l(\O_{\S'})$ satisfies the claimed bound by \cref{l:sharp-bound-length}. 
Since $l(\O_\S) \ge l(\O_{\S'})$, this proves the proposition.
\end{proof}

\subsection{Mather log dicrepancies}

A different way of measuring singularities 
of pairs on singular varieties is to use the Jacobian ideal sheaf $\Jac_f = \Fitt^0(\Om_{Y/X})$
of a log resolution $f \colon Y \to X$ in place of the relative canonical divisor. 
We define the \emph{Mather log discrepancy} of $E$ over a pair $(X,Z)$ to be 
\[
\^a_E(X,Z) := \ord_E(\Jac_f) + 1 - \ord_E(Z).
\]
If $X$ is smooth then $\^a_E(X,Z) = a_E(X,Z)$, but we have a strict inequality $\^a_E(X,Z) > a_E(X,Z)$
as soon as $X$ is singular at the generic point of the center of $E$. 
This follows from \cite[Proposition~3.4]{dFD14}, 
which shows that $\ord_E(\Jac_X) = \ord_E(K_{Y/X}) + \frac 1r \ord_E(\fn_{r,X})$
where $\Jac_X = \Fitt^{\dim X}(\Om_X)$ is the Jacobian ideal sheaf of $X$, 
$r$ is any positive integer such that $rK_X$ is Cartier, and
$\fn_{r,X}$ is the ideal sheaf defined by the image of the natural map
$(\wedge^{\dim X}\Om_X)^{\otimes r} \to \O_X(rK_X)$, 
and Nobile's theorem \cite[Theorem~2]{Nob75}, which implies that 
$\fn_{r,X}$ vanishes precisely on the singular locus of $X$. 

\begin{example}
If $X$ has locally complete intersection singularities, then 
$\ord_E(\Jac_f) = \ord_E(K_{Y/X}) + \ord_E(\Jac_X)$ (see \cite[Proposition~1]{Pie79}), hence
\[
\^a_E(X,Z) = a_E(X,Z) + \ord_E(\Jac_X).
\]
\end{example}

Let now $X \subset \A^N$ be a Cohen--Macaulay variety of dimension $n$,
and $E$ a divisor over $X$.
Consider general linear projections
\[
X \subset \A^N \xrightarrow{\ff} U=\A^n \xrightarrow{\g} V = \A^m,
\]
where $1 \le m \le n$.
We denote by $\f \colon X \to V$ be the composition. 
Write $\val_E|_{k(U)} = p \val_F$ and $\val_E|_{k(V)} = q \val_G$ where
$F$ and $G$ are divisors over $U$ and $V$, respectively,  
and $p$ and $q$ are positive integers (e.g., see \cite[Lemma~2.3]{dFM15}). 
Let $Z \subset X$ be a closed subscheme of codimension $r = n-m+1$ cut out by a regular sequence.
We assume that $\f|_Z$ is a proper finite morphism.
Note that $\f_*[Z]$ is a cycle of codimension 1 in $V$.
We regard $\f_*[Z]$ as a Cartier divisor on $V$. 

\begin{theorem}
\label{t:projection}
With the above notation, for every $c \ge 0$ such that $\^a_E(X,cZ)\geq 0$ we have
\[
q\, a_G(V,\tfrac {c^r}{r^r}\.\f_*[Z])  \le \^a_E(X,cZ).
\]
\end{theorem}

This theorem was first proved in \cite{dFEM03} assuming that $X$ is smooth.
The proof was then extended to singular varieties using Mather log discrepancies in \cite{dFM15}.

\section{Birational rigidity}
\label{s:smooth}

A \emph{Mori fiber space} is a normal projective variety $X$ with $\Q$-factorial terminal 
singularities, equipped with an extremal Mori contraction $f \colon X \to S$ of fiber type.  
This means that $f$ is a proper morphism with connected fibers and relative Picard number 1, 
the anticanonical class $-K_X$ is $f$-ample, and $\dim S < \dim X$. 
Note that Fano varieties with $\Q$-factorial terminal singularities 
and Picard number 1 can be viewed as Mori fiber spaces over a point.

A Mori fiber space $f\colon X \to S$ is \emph{birationally superrigid}
if for every birational map $\f \colon X \rat X'$ from $X$ to another Mori fiber space $f' \colon X' \to S'$
there exists a birational map $\ff \colon S \rat S'$ such that $f'\o \f = \ff \o f$
and, furthermore, $\f$ induces an isomorphism between the generic fibers of $f$ and $f'$. 

Birational superrigidity is a notion that is naturally motivated from the point of view of the minimal
model program. Clearly, birational superrigid varieties are nonrational. 
In dimension 2, every rationally connected Mori fiber space is rational and hence 
cannot be birationally superrigid, 
something that should be contrasted with the fact that two-dimensional minimal models are unique
in their birational classes. Both facts are no longer true in higher dimensions.
The first example of birationally superrigid rationally connected Mori fiber space 
was discovered as a byproduct 
of Iskovskikh and Manin's proof of nonrationality of smooth quartic threefolds in $\P^4$. 

Restricting to a Fano variety $X$ with $\Q$-factorial terminal singularities 
and Picard number 1, we see that $X$ is birationally superrigid if and only if every 
birational map $\f \colon X \rat X'$ to a Mori fiber space is an isomorphism. 
The following result, established in \cite{IM71,Cor00}, is at the center of the study of birationally superrigid 
Fano varieties. 
A more general version can be stated where $X$ is any Mori fiber space, 
but we will not need it.

\begin{theorem}[Noether--Fano Inequality]
\label{t:NoetherFano}
Let $X$ be a Fano variety with $\Q$-factorial terminal singularities and Picard number 1, and let
$\f \colon X \rat X'$ be a birational map to a Mori fiber space $f' \colon X' \to S'$. 
Fix a projective embedding of $X'$ given by a linear system $\cH' \subset |-r'K_{X'} + A'|$
where $r'$ is a positive rational number and $A'$ is the pullback of an ample divisor on $S'$
(we set $A' = 0$ if $S'$ is a point), so that $\f$ is defined by the movable linear
system $\cH := \f_*^{-1}\cH'$. Let $r$ the positive rational number such that 
$\cH \subset |-rK_X|$. 
If $\ct(X,\cH) \le \frac 1r$, then $\f$ is an isomorphism.
\end{theorem}

Before addressing \cref{t:ODP}, we quickly revisit the proof of the smooth case.
We follow the proof given in \cite{Zhu20,Kol19}.

\begin{proof}[Proof of \cref{t:smooth}]
Let $X \subset \P^{n+1}$ be a smooth hypersurface of degree $n+1$, with $n \ge 3$. 
Note that $-K_X \sim \O_X(1)$, and this generates 
the Picard group of $X$ by the Lefschetz hyperplane theorem.

We assume by contradiction that $X$ is not birationally superrigid, 
and let $\f$ be a birational map 
to a Mori fiber space $X'$ that is not an isomorphism. 
Let $\cH \subset |-rK_X|$ be a linear system defining $\f$ as in \cref{t:NoetherFano}. 
By the Noether--Fano inequality, we have
\[
c:= \ct(X,\cH) < \frac 1r.
\] 
On the other hand, \cref{p:Puk} implies that 
$e_C(D) \le r$ for every $D \in \cH$ and every irreducible 
curve $C \subset X$, hence it follows by \cref{t:mult-thresholds} that
$(X,\frac 1r \cH)$ is canonical in dimension 1 (i.e., away from a finite set).
Therefore there is a closed point $p \in X$ such that 
\[
\mld_p(X,c\cH) = 1. 
\]
Fix now two general elements $D,D' \in \cH$, and let $Z = D \cap D' \subset X$
their schematic intersection. For any subvariety $V \subset X$, we will denote by $Z|_V$
the intersection $Z \cap V$. 
Note that we still have $\mld_p(X,cZ) = 1$, and
\cref{p:Puk} implies that the set $\{x \in X \mid e_x(Z) > r^2 \}$
is at most one-dimensional.
Let $Y \subset X$ be a general hyperplane section through $p$. 
Bertini's theorem ensures that the pair $(Y,cZ|_Y)$ is canonical away from $p$, and
inversion of adjunction (\cref{t:inv-of-adj}) implies that
\[
\mld_p(Y,cZ|_Y) \le 0.
\]
Note that, by \cref{p:Bertini}, the set $\dim\{y \in Y \mid e_y(Z|_Y) > r^2 \}$ is zero-dimensional.

If $n = 3$, then $Y$ is a surface and $Z|_Y$ is zero dimensional. 
Since $\lct_p(Y,Z|_Y) < \frac 1r$, we have
$e_{Z|_Y}(Y)_p > 4r^2$ by~\cref{t:mult-thresholds}. 
On the other hand, observing that $Y$ has degree 4 and $Z|_Y$ is a zero dimensional complete intersection scheme
cut out on $Y$ by two equations of degree $r$, we also have 
$e_{Z|_Y}(Y)_p \le 4r^2$ by \cref{p:Bezout1}. This gives a contradiction.

Assume therefore that $n \ge 4$. 

We clam that the pair $(Y,2cZ|_Y)$ is log terminal in dimension 1.
To see this, first recall that the set
$\{y \in Y \mid e_y(Z|_Y) > r^2 \}$ is zero-dimensional, hence there is a 
finite set $\G \subset Y$ such that 
$e_q(Z|_Y) \le r^2$ for all $q \in Y \setminus \G$. 
For any such $q$, let $S \subset Y$ be a smooth surface cut out by general hyperplanes through $q$.
We have $e_q(Z|_S) \le r^2$, 
and since $Z|_S$ is a zero-dimensional complete intersection subscheme of $S$, we have
\[
\lct_q(S,Z|_S) \ge \frac{2}{\sqrt{e_q(Z|_S)}} \ge \frac 2r > 2c
\]
by \cref{t:mult-thresholds} (cf.\ \cref{eg:e_pZ=e_ZX_p}).
It follows that $(S,2cZ|_S)$ is log terminal near $q$. By inversion of adjunction, 
this implies that $(Y,2cZ|_Y)$ is log terminal near $q$. 
This proves that $(Y,2cZ|_Y)$ is log terminal away from a finite set $\G$, as claimed. 

Therefore the multiplier ideal $\J(Y,2cZ|_Y)$
defines a zero-dimensional subscheme $\S \subset Y$. 
Since $Z|_Y$ is cut out by forms of degree $r$ and $2cr < 2$, we have 
$H^1(Y,\O_Y(2) \otimes \J(Y,2cZ|_Y)) = 0$ 
by Nadel's vanishing theorem (\cref{t:Nadel}).
This implies that there is a surjection
$H^0(Y,\O_Y(2)) \surj  H^0(\S,\O_\S)$.
Keeping in mind that $H^0(Y,\O_Y(2)) \cong H^0(\P^{n},\O_{\P^{n}}(2))$, we obtain
\[
h^0(\S,\O_\S) \le h^0(Y,\O_Y(2)) = \binom{n+2}{2}.
\]
On the other hand, \cref{p:colength-mult-ideal} gives the lower-bound
\[
l(\O_\S) \ge \tfrac 12 3^{n-1}.
\]
The two inequalities lead to a contradiction as soon as $n \ge 5$. 

With the case $n=3$ already settled, this leaves open only the case $n=4$. 
We treat this case using generic projections, as in \cite{dFEM03}. 
Note that the same argument can also be used to deal with the case $n=3$, 
instead of the argument based on Bezout's theorem we outlined earlier,
hence one can think of the proof of \cref{t:smooth} 
as splitting into two parts rather than three.

If $n=4$, then $Y$ is a threefold in $\P^4$ and $Z|_Y$ is one dimensional. 
Let $f \colon Y \rat \P^2$ be the map induced by 
a general linear projection $\P^4 \rat \P^2$. 
We may assume that the indeterminacies of $f$ are disjoint from the support of $Z|_Y$, hence
we can define the $\Q$-divisor
\[
\D = \frac{c^2}{4}\. f_*[Z|_Y].
\]
Since the set $\{y \in Y \mid e_y(Z|_Y) > r^2 \}$ is zero dimensional and, for a general projection, 
$f$ restricts to a birational morphism on the support of $Z|_Y$, it
follows that the pair $(\P^2,\D)$ is log terminal in dimension one. By contrast, 
\cref{t:projection} implies that the pair is not log terminal at $f(p)$. 
Therefore, $\J(\P^2,\D)$ defines a zero dimensional scheme $W \subset \P^2$.
Note that $\deg(\D) < 2$. We have $H^1(\P^2,\O(-1) \otimes \J(\P^2,\D)) = 0$ by Nadel's vanishing theorem, 
and this yields a surjection $H^0(\P^2,\O(-1)) \surj H^0(\O_W)$, 
which is impossible since $H^0(\O_W) \ne \{0\}$.
\end{proof}

\begin{proof}[Proof of \cref{t:ODP}]
Let $X \subset \P^{n+1}$ be a hypersurface
of degree $n+1$ with only isolated ordinary double points as singularities, and assume that $n \ge 5$. 
Note that $X$ is a Fano variety with terminal $\Q$-factorial singularities 
and $K_X$ generates the Picard group. 
It is easy to see that ordinary double points in dimension $\ge 3$
are terminal and locally factorial in the analytic topology, 
but $\Q$-factoriality is a local property in the Zariski topology, so it needs to be
verified. This is done, for instance, in \cite[Lemma~5.1]{dF17}. 

The starting point of the proof of \cref{t:ODP} is the same as in the smooth case just discussed.
Assuming $X$ is not birationally superrigid, we construct a
movable linear system $\cH \subset |-rK_X|$ such that
\[
c:= \ct(X,\cH) < \frac 1r.
\] 
Let $E$ be an exceptional divisor over $X$ such that $a_E(X,c\cH) = 1$, and 
let $T \subset X$ be the center of $E$. 
We fix a general point $p \in T$.
It follows by \cref{p:Puk} and \cref{t:mult-thresholds} (see also \cref{r:Puk}) that
$T$ is at most one-dimensional, and it is zero-dimensional 
if it is contained in the smooth locus of $X$. 
If $T \subset X_\reg$, then we can argue as in the proof of \cref{t:smooth} to get a contradiction.
We can therefore assume that $T \not\subset X_\reg$. 
 
Let $Z = D \cap D'$ be the intersection of two general members $D,D' \in \cH$ and
$Y = V \cap V' \subset X$ 
the intersection of two general hyperplane sections $V,V' \subset X$ through $p$. 
After cutting down with one hyperplane, we get a pair $(V,cZ|_V)$ that is canonical in dimension 1. 
After cutting down with the second hyperplane, 
we get a pair $(Y,cZ|_Y)$ that is not log canonical at $p$; here we apply inversion of adjunction.
Furthermore, by \cref{p:Bertini,p:Puk} (see also \cref{r:Puk})
we can ensure that the set $\{y \in Y \mid e_y(Z|_Y) > r^2 \}$ is zero-dimensional.

We split the proof in two cases, depending on the dimension of $X$.

We first treat the case where $n \ge 8$. 
By the same argument as in the proof of \cref{t:smooth}, we see that
$\J(Y,2cZ|_Y)$ defines a zero dimensional scheme $\S \subset Y$.
Since $Z|_Y$ is cut out on $Y$ by forms of degree $r$ and $2cr < 2$, we have 
$H^1(Y,\O_Y(3) \otimes \J(Y,2cZ|_Y)) = 0$ by Nadel vanishing, hence
\begin{equation}
\label{eq:upperbound}
h^0(\O_\S) \le h^0(Y,\O_Y(3)) = h^0(\P^{n-1},\O(3)) = \binom{n+2}{3}.
\end{equation}

By \cref{l:double-log-discr}, we have $\lct(Y,\S) \le 1$.
If $p$ is a smooth point, then we can apply~\cref{l:sharp-bound-length}
to get a lower bound. However, if $p$ is a singular point then we cannot apply the bound directly.
Instead, we will take a suitable degeneration which will allow
us to apply~\cref{l:sharp-bound-length} in lower dimension. 
 
Let us discuss here the case where $p$ is a singular point, 
the other case being easier and leading to a stronger bound. 
We restrict to affine chart $\A^{n-1}$ containing $p$, 
and fix coordinates $(x_1,\dots,x_{n-1})$ centered at $p$
such that the tangent cone of $Y$ at $p$ is defined by $\sum x_i^2 = 0 $. 
For simplicity, we still denote by $Y$ its restriction to $\A^{n-1}$ and let $\S$ denote now the connected
component of the zero dimensional scheme defined by $\J(Y,2cZ|_Y)$ that is supported at $p$. 

We take a flat degeneration of $Y \subset \A^{n-1}$ to the union $H_1 \cup H_2$ of two
hyperplanes through $p$, and let $\S' \subset \A^{n-1}$ be the zero dimensional 
scheme supported at $p$ obtained by flat degeneration from $\S$. Note that $\S' \subset H_1 \cup H_2$. 
Concretely, if $f(x_1,\dots,x_{n-1}) = 0$ is the equation of $Y$ in $\A^{n-1}$, then 
degeneration is constructed using the ${\mathbb G}_m$-action
on $\A^{n-1}$ given by $x_i \mapsto \lambda x_i$ for $i=1,2$
and $x_j \mapsto \lambda^2 x_j$ for $j > 2$, and sending $\lambda \to 0$. 

Recall that $\mld_p(Y,\S) \le 0$. 
By adjunction, this implies that 
\[
\mld_p(\A^{n-1},\S + Y) \le 0.
\]
By semicontinuity of log canonical thresholds, we have
\[
\mld_p(\A^{n-1},\S' + H_1 + H_2) \le 0.
\]  
By inversion of adjunction, this implies that
\[
\mld_p(H_1,\S'_1 + H_{12}) \le 0
\]
where $\S'_1 = \S' \cap H_1$ and $H_{12} = H_1 \cap H_2$. 
Applying inversion of adjunction again, we get
\[
\mld_p(H_{12},\S'_{12}) \le 0
\]
where $\S'_{12} = \S'_1 \cap H_{12} = \S' \cap H_{12}$. 
Note that $H_{12} = \A^{n-3}$. 
Then~\cref{l:sharp-bound-length} implies that
\[
l(\O_{\S'_{12}}) \ge \tfrac 12 3^{n-3}.
\]
Since $l(\O_\S) = l(\O_{\S'})$ by flatness and
$\l(\O_{\S'}) \ge l(\O_{\S'_{12}})$ from the inclusion $\S'_{12} \subset \S'$, 
we conclude that
\begin{equation}
\label{eq:lowerbound}
l(\O_\S) \ge \tfrac 12 3^{n-3}.
\end{equation}
Comparing the two inequalities \eqref{eq:upperbound}
and \eqref{eq:lowerbound}, we get a contradiction as soon as $n \ge 8$. 
So the theorem is proved in this case. 

We now address the remaining cases $5 \le n \le 7$.
Suppose first that $\dim T = 1$. 
Recall that we have fixed a general point $p \in T$ 
and let $Y := V \cap V' \subset X$ be
the intersection of two general hyperplane sections through $p$. 
Let $f \colon Y \rat \P^{n-3}$ be the map induced by 
a general linear projection $\P^{n-1} \rat \P^{n-3}$, and let
\[
\D = \frac{c^2}{4}\. \p_*[Z|_Y].
\]
This is an effective $\Q$-divisor of degree less than $(n+1)/4$.  
Since $\deg(\D) < 2$, Nadel's vanishing theorem gives
\[
H^1(\P^{n-3},\O(4-n) \otimes \J(\P^{n-3},\D)) = 0.
\]
\cref{t:projection} implies that the pair $(\P^{n-3},\D)$ is not log terminal at $f(p)$, and
since the set $\{y \in Y \mid e_y(Z|_Y) > r^2 \}$ is zero dimensional, we can 
argue as in the proof of \cite[Theorem~4.1]{dFEM03} that
if $n \in \{5,6\}$ then $\J(\P^{n-3},\D)$ defines a zero dimensional scheme $W \subset \P^{n-3}$,
hence the surjection
\[
H^0(\P^{n-3},\O(4-n)) \surj H^0(\O_W)
\]
forces $n \le 4$. 
If $n=7$, then one can only conclude that $\J(\P^4,\D)$ defines a scheme of dimension
at most 1, but after cutting down one more time and using the vanishing
$H^1(\P^3,\O(-2) \otimes \J(\P^3,\D|_{\P^3})) = 0$
we get to a similar contradiction. 

Therefore $T$ must be zero dimensional, i.e., $T = \{p\}$.
With the case $p \in X_\reg$ already been settled by the proof of \cref{t:smooth}, 
we assume that $p$ is a singular point. 
After cutting down with just one hyperplane section $V \subset X$ through $p$, 
we get $\mld_p(V,cZ|_V) \le 0$ by inversion of adjunction, hence 
we can find a divisor $E$ over $V$ with center $p$ such that
$a_E(V,cZ|_V) \le 0$. Note that $\Jac_V\.\O_{V,p} = \fm_{V,p}$ since $p$ is an ordinary double point, 
hence 
\[
\^a_E(V,cZ|_V) \le \val_E(\fm_{V,p}). 
\]
We take a general linear projection $g \colon V \to \P^{n-2}$ and let
\[
\Theta = \frac{c^2}{4}\. \p_*[Z|_V].
\]
By \cref{t:projection}, there exists a divisor $G$ over $\P^{n-2}$ with center $q = g(p)$ such that, 
\[
a_G(\P^{n-2},\Theta) \le \val_G(\fm_q).
\]
Restricting to a general hyperplane $\P^{n-3}$ through $q$
and letting $\D$ be the restriction of $\Theta$, we get
a pair $(\P^{n-3},\D)$ which satisfies the same properties
as in the situation we discussed in the case $\dim T = 1$. 
We can therefore repeat the same argument to get a contradiction.

This completes the proof of \cref{t:ODP}. 
\end{proof}

\begin{remark}
\label{r:rank>1}
The first part of the proof of \cref{t:ODP} 
works for more general quadratic singularities, with
the same argument showing that, for $n \ge 8$, 
every hypersurface of degree $n+1$ in $\P^{n+1}$
with isolated quadratic singularities of rank $\ge 2$ is birationally superrigid. 
Note, however, that the part of the proof of \cref{t:ODP} dealing with the case 
$5 \le n \le 7$ does not extend automatically,
as the Jacobian ideal at a quadratic singularity of submaximal rank is
strictly contained in the maximal ideal and can have larger order along some valuations, 
invalidating the last step of the proof. 
\end{remark}


\section{K-stability}

K-stability is an algebro-geometric notion related to the existence of
K\"ahler--Einstein metrics on Fano varieties. 
The original definition is given in terms of 
the positivity of the generalized Futaki invariant of test configurations. 
We will recall here an equivalent definition 
following the valuative approach of \cite{Fuj19,Li17}. 
According to the latter approach,
an $n$-dimensional normal $\Q$-Gorenstein projective variety $X$ with ample anticanonical class 
is \emph{K-stable} if for every divisor $E$ over $X$
\[
a_E(X) > \frac 1{(-K_X)^n} \int_0^\infty \vol(f^*(-K_X)-tE)\,dt,
\]
where $f \colon Y \to X$ is a resolution on which $E$ appears as a divisor. 
This definition better fits the theme of this paper. 
The equivalence between this definition and the one via 
the generalized Futaki invariant is proved in \cite{BX19}.
We refer to \cite{Xu20} for a general introduction to the subject. 

We use the following result, due to \cite{SZ19}.
A stronger result along the same lines is proved in \cite{Zhu20}, but we will not need it.

\begin{theorem}
\label{t:SZ}
Let $X$ be a Fano variety with $\Q$-factorial terminal singularities and Picard number 1. 
Assume that $X$ is birationally superrigid and $\lct(X,D) > \frac 12$
for every effective $\Q$-divisor $D \sim_\Q -K_X$.
Then $X$ is K-stable. 
\end{theorem}

As it is already observed in \cite{SZ19}, in view of \cref{t:smooth} and
well-established bounds on the alpha invariant of smooth Fano hypersurfaces, 
this result implies immediately that smooth hypersurfaces of 
degree $n+1$ in $\P^{n+1}$ are $K$-stable hence admit a K\"ahler--Einstein metric. 
In order to apply the result to hypersurfaces with isolated ordinary double points, 
we need the following property. 

\begin{proposition}
\label{p:alpha}
For any $n \ge 4$, if $X \subset \P^{n+1}$ is a hypersurface of degree $n+1$
with only isolated ordinary double points as singularities, then 
$\lct(X,D) > \frac 12$ for every effective $\Q$-divisor $D \sim_\Q -K_X$.
\end{proposition}

\begin{proof}
Suppose by contradiction that there exists an effective 
$\Q$-divisor $D \sim_\Q -K_X$ such that $\lct(X, D) \le \frac 12$.
Let $E$ be a divisor over $X$ such that $a_E(X,\frac 1{2} D) \le 0$, 
and let $T \subset X$ be the center of $E$ in $X$. 
\cref{t:mult-thresholds,p:Puk} imply that $T$ is at most one-dimensional, 
and that it is zero-dimensional if it is contained in the smooth locus of $X$. 
Let $p \in T$ be a general point.

Let $f \colon X \to \P^n$ be induced by a general linear projection, let $q = f(p)$, and let
\[
\D = \frac 1{2}\. f_*D.
\]
We fix a general hyperplane $\P^{n-1} \subset \P^n$ through $q$. 

We claim that
\begin{equation}
\label{eq:mld_q}
\mld_q(\P^{n-1},\D|_{\P^{n-1}}) \le 0.
\end{equation}
If $\dim T = 1$ and $\x$ is the generic point of $f(T)$, then we have $\mld_\x(\P^n,\D) \le 0$ 
by \cref{t:projection}, and since $q$ is a general point of $f(T)$, 
\eqref{eq:mld_q} follows by Bertini's theorem applied on a log resolution. 
If $T=\{p\}$ is zero-dimensional, then we have $\^a_E(X,\frac 12 D) \le \val_E(\fm_p)$.
This is clear if $p \in X_\reg$, and holds if $p$ is a singular point since
in this case we have $\Jac_X\.\O_{X,p} = \fm_p$. 
Then \cref{t:projection} implies that there is a divisor $F$ over $\P^n$
with center $q$ such that $a_F(\P^n,\D) \le \val_F(\fm_q)$. 
This implies that $\mld_q(\P^n,\D+ \P^{n-1}) \le 0$, hence \eqref{eq:mld_q}
follows by inversion of adjunction. 

Our next claim is that the pair $(\P^{n-1},\D|_{\P^{n-1}})$ is log terminal in dimension 1 near $q$. 
We will prove this by showing that the pair $(\P^n,\D)$ is log terminal in dimension 2 near $q$. 
To see this, we restrict to formal neighborhoods. 
Let $U = \Spec(\^{O_{X,p}})$ and $V = \Spec(\^{\O_{\P^n,q}})$, 
and denote by $g = f|_U \colon U \to V$ the induced map. 
Resolution of singularities holds in this setting (and in fact for the 
purpose of this proof it suffice to rely solely on resolutions obtained by base-change from 
$X$ and $\P^n$, respectively), so we can extend the definition of singularities to this setting. 
We are going to show that $(V,\D|_V)$ is log terminal in dimension 2. 
By taking a general projection we can ensure that $p$ is the 
only point in the support of $D$ mapping to $q$, hence 
$\D|_V = \frac 1{2} \.g_*(D|_U)$.
If $p$ is a smooth point of $X$, then $g$ is an isomorphism
and the claim is clear. 
If $p$ is a singular point, then $g$ is a generically two-to-one cover. 
In this case, if $G$ is any divisor over $V$ with positive dimensional center $W \subset V$,
and $G'$ is any divisor over $U$ such that $\val_{G'}$ restricts to 
$b\val_G$ on $V$ for some $b \in \N$ (such $G'$ always exists), then 
a standard computation gives
\begin{equation}
\label{a_G(V,D)}
b \, a_G(V,\D|_V) = a_{G'}(U,g^*(\D|_V) - K_{U/V}) \ge a_{G'}(U,g^*(\D|_V)). 
\end{equation}
Here we are using the fact that $W$ is positive dimensional in order to reduce to a computation 
on smooth varieties (i.e., away from $p$ and $q$), and use that $K_{U/V} \ge 0$ in the last inequality. 
If $\t \colon U \to U$ is the involution defined by the cover $U \to V$, 
then $g^*(\D|_V) = \tfrac{1}{2}(D|_U + \t(D|_U))$, hence
\[
\val_{G'}(g^*(\D|_V)) \le \max \{\val_{G'}(D|_U), \val_{G'}(\t(D|_U))\}.
\]
Since both $(U,D|_U)$ and $(U,\t(D|_U))$ are log terminal 
in dimension 2, it follows that $a_{G'}(U,g^*(\D|_V)) > 0$, and hence $a_G(V,\D|_V) > 0$ by \eqref{a_G(V,D)}, 
whenever $\dim(W) \ge 2$. This proves that $(V,\D|_V)$ is log terminal in dimension 2, which in turns implies
that $(\P^n,\D)$ is log terminal in dimension 2 near $q$.
Since $\P^{n-1}$ is a general hyperplane through $q$, 
we conclude that $(\P^{n-1},\D|_{\P^{n-1}})$ is log terminal in dimension 1 near $q$. 

We are now ready to finish the proof. 
Since $\deg(\D)  < \ru{\frac n2} + 1$, Nadel's vanishing gives
\[
H^1(\P^{n-1},\O(-n + \ru{\tfrac n2} + 1) \otimes \J(\P^{n-1},\D|_{\P^{n-1}})) = 0.
\]
Note that $-n + \ru{\tfrac n2} + 1 \le -1$ for $n \ge 4$. 
As the multiplier ideal $\J(\P^{n-1},\D|_{\P^{n-1}})$ defines a scheme with a zero-dimensional component
supported at $q$, this gives a contradiction. 
\end{proof}

\begin{proof}[Proof of \cref{t:ODP-K-stab}]
It follows by \cref{p:alpha,t:SZ} that $X$ is K-stable, and \cite[Theorem~1.5]{LXZ21}
implies that $X$ admits a weak K\"ahler--Einstein metric. 
\end{proof}

\begin{remark}
\cref{p:alpha} also holds for $n=4$, a case which is not covered by \cref{t:ODP}.
If one can extend \cref{t:ODP} to include the four-dimensional case, then
K-stability would follow as well. 
\end{remark}


\begin{bibdiv}
\begin{biblist}


\bib{Ben70}{article}{
   author={Bennett, B. M.},
   title={On the characteristic functions of a local ring},
   journal={Ann. of Math.},
   volume={91},
   date={1970},
   pages={25--87},
}

\bib{BX19}{article}{
   author={Blum, Harold},
   author={Xu, Chenyang},
   title={Uniqueness of K-polystable degenerations of Fano varieties},
   journal={Ann. of Math. (2)},
   volume={190},
   date={2019},
   number={2},
   pages={609--656},
}

\bib{Che01}{article}{
   author={Cheltsov, Ivan},
   title={Log canonical thresholds on hypersurfaces},
   language={Russian, with Russian summary},
   journal={Mat. Sb.},
   volume={192},
   date={2001},
   number={8},
   pages={155--172},
   translation={
      journal={Sb. Math.},
      volume={192},
      date={2001},
      number={7-8},
      pages={1241--1257},
   },
}

\bib{Che07}{article}{
   author={Cheltsov, Ivan},
   title={On nodal sextic fivefold},
   journal={Math. Nachr.},
   volume={280},
   date={2007},
   number={12},
   pages={1344--1353},
}



\bib{Cor00}{article}{
   author={Corti, Alessio},
   title={Singularities of linear systems and $3$-fold birational geometry},
   conference={
      title={Explicit birational geometry of 3-folds},
   },
   book={
      series={London Math. Soc. Lecture Note Ser.},
      volume={281},
      publisher={Cambridge Univ. Press, Cambridge},
   },
   date={2000},
   pages={259--312},
}

\bib{CM04}{article}{
   author={Corti, Alessio},
   author={Mella, Massimiliano},
   title={Birational geometry of terminal quartic 3-folds. I},
   journal={Amer. J. Math.},
   volume={126},
   date={2004},
   number={4},
   pages={739--761},
}

\bib{dF13}{article}{
   author={de Fernex, Tommaso},
   title={Birationally rigid hypersurfaces},
   journal={Invent. Math.},
   volume={192},
   date={2013},
   number={3},
   pages={533--566},
}

\bib{dF16}{article}{
   author={de Fernex, Tommaso},
   title={Erratum to: Birationally rigid hypersurfaces},
   journal={Invent. Math.},
   volume={203},
   date={2016},
   number={2},
   pages={675--680},
}

\bib{dF17}{article}{
   author={de Fernex, Tommaso},
   title={Birational rigidity of singular Fano hypersurfaces},
   journal={Ann. Sc. Norm. Super. Pisa Cl. Sci. (5)},
   volume={17},
   date={2017},
   number={3},
   pages={911--929},
}

\bib{dFD14}{article}{
   author={de Fernex, Tommaso},
   author={Docampo, Roi},
   title={Jacobian discrepancies and rational singularities},
   journal={J. Eur. Math. Soc. (JEMS)},
   volume={16},
   date={2014},
   number={1},
   pages={165--199},
}

		
\bib{dFEM03}{article}{
   author={de Fernex, Tommaso},
   author={Ein, Lawrence},
   author={Musta\c{t}\u{a}, Mircea},
   title={Bounds for log canonical thresholds with applications to
   birational rigidity},
   journal={Math. Res. Lett.},
   volume={10},
   date={2003},
   number={2-3},
   pages={219--236},
}

\bib{dFEM04}{article}{
   author={de Fernex, Tommaso},
   author={Ein, Lawrence},
   author={Musta\c{t}\u{a}, Mircea},
   title={Multiplicities and log canonical threshold},
   journal={J. Algebraic Geom.},
   volume={13},
   date={2004},
   number={3},
   pages={603--615},
}	

\bib{dFM15}{article}{
   author={de Fernex, Tommaso},
   author={Musta{\c{t}}{\u{a}}, Mircea},
   title={The volume of a set of arcs on a variety},
   journal={Rev. Roumaine Math. Pures Appl.},
   volume={60},
   date={2015},
   pages={375--401},
   note={Special issue in honor of Lucian Badescu's 70th birthday},
}

\bib{EP14}{article}{
   author={Eckl, Thomas},
   author={Pukhlikov, Aleksandr},
   title={On the locus of nonrigid hypersurfaces},
   conference={
      title={Automorphisms in birational and affine geometry},
   },
   book={
      series={Springer Proc. Math. Stat.},
      volume={79},
      publisher={Springer, Cham},
   },
   date={2014},
}

\bib{Fan07}{article}{
   author={Fano, Gino},
   title={Sopra alcune variet\`a algebriche a tre dimensioni aventi tutti i generi nulli},
   journal={Atti Accad. Torino},
   volume={43},
   date={1907/08},
   pages={973--377},
}

\bib{Fan15}{article}{
   author={Fano, Gino},
   title={Osservazioni sopra alcune variet\`a non razionali aventi tutti i generi nulli},
   journal={Atti Accad. Torino},
   volume={47},
   date={1915},
   pages={1067--1071},
}

\bib{Fuj19}{article}{
   author={Fujita, Kento},
   title={A valuative criterion for uniform K-stability of $\Bbb Q$-Fano
   varieties},
   journal={J. Reine Angew. Math.},
   volume={751},
   date={2019},
   pages={309--338},
}

\bib{Fuj19b}{article}{
   author={Fujita, Kento},
   title={K-stability of Fano manifolds with not small alpha invariants},
   journal={J. Inst. Math. Jussieu},
   volume={18},
   date={2019},
   number={3},
   pages={519--530},
}

\bib{Ful98}{book}{
   author={Fulton, William},
   title={Intersection theory},
   series={Ergebnisse der Mathematik und ihrer Grenzgebiete. 3. Folge. A
   Series of Modern Surveys in Mathematics [Results in Mathematics and
   Related Areas. 3rd Series. A Series of Modern Surveys in Mathematics]},
   volume={2},
   edition={2},
   publisher={Springer-Verlag, Berlin},
   date={1998},
} 


\bib{IM71}{article}{
   author={Iskovskih, V. A.},
   author={Manin, Ju. I.},
   title={Three-dimensional quartics and counterexamples to the L\"{u}roth
   problem},
   language={Russian},
   journal={Mat. Sb. (N.S.)},
   volume={86(128)},
   date={1971},
   pages={140--166},
}


\bib{Kol92}{article}{
   author={Koll\'ar, J\'anos},
   title={Adjunction and discrepancies},
   conference={
   title={Flips and abundance for algebraic threefolds},
   },
   note={Papers from the Second Summer Seminar on Algebraic Geometry held at
   the University of Utah, Salt Lake City, Utah, August 1991;
   Ast\'{e}risque No. 211 (1992) (1992)},
   publisher={Soci\'{e}t\'{e} Math\'{e}matique de France, Paris},
   date={1992},
   pages={1--258},
}

\bib{Kol19}{article}{
   author={Koll\'ar, J\'anos},
   title={The rigidity theorem of Fano--Segre--Iskovskikh--Manin--Pukhlikov--Corti--Cheltsov--de Fernex--Ein--Musta\c t\u a--Zhuang},
   conference={
      title={Birational Geometry of Hypersurfaces, Gargnano del Garda, Italy, 2018},
   },
   book={
      series={Lecture Notes of the Unione Matematica Italiana},
      publisher={Springer International Publishing},
   },
   date={2019},
}

\bib{Laz04}{book}{
   author={Lazarsfeld, Robert},
   title={Positivity in algebraic geometry. II},
   series={Ergebnisse der Mathematik und ihrer Grenzgebiete. 3. Folge. A
   Series of Modern Surveys in Mathematics [Results in Mathematics and
   Related Areas. 3rd Series. A Series of Modern Surveys in Mathematics]},
   volume={49},
   note={Positivity for vector bundles, and multiplier ideals},
   publisher={Springer-Verlag, Berlin},
   date={2004},
}

\bib{Li17}{article}{
   author={Li, Chi},
   title={K-semistability is equivariant volume minimization},
   journal={Duke Math. J.},
   volume={166},
   date={2017},
   number={16},
   pages={3147--3218},
}

\bib{LZ19}{article}{
   author={Liu, Yuchen},
   author={Zhuang, Ziquan},
   title={On the sharpness of Tian's criterion for K-stability},
   note={Preprint, {\tt arXiv:1903.04719}},
   date={2019},
}

\bib{LZ21}{article}{
   author={Liu, Yuchen},
   author={Zhuang, Ziquan},
   title={Birational superrigidity and $K$-stability of singular Fano
   complete intersections},
   journal={Int. Math. Res. Not. IMRN},
   date={2021},
   number={1},
   pages={384--403},
}

\bib{LXZ21}{article}{
   author={Liu, Yuchen},
   author={Xu, Chenyang},
   author={Zhuang, Ziquan},
   title={Finite generation for valuations computing stability thresholds and applications to K-stability},
   note={Preprint, {\tt arXiv:2102.09405}},
   date={2021},
}

\bib{Man66}{article}{
   author={Manin, Ju. I.},
   title={Rational surfaces over perfect fields},
   language={Russian, with English summary},
   journal={Inst. Hautes \'{E}tudes Sci. Publ. Math.},
   number={30},
   date={1966},
   pages={55--113},
}

\bib{Mel04}{article}{
   author={Mella, Massimiliano},
   title={Birational geometry of quartic 3-folds. II. The importance of
   being $\mathbb Q$-factorial},
   journal={Math. Ann.},
   volume={330},
   date={2004},
   number={1},
   pages={107--126},
}

\bib{Nob75}{article}{
   author={Nobile, A.},
   title={Some properties of the Nash blowing-up},
   journal={Pacific J. Math.},
   volume={60},
   date={1975},
   number={1},
   pages={297--305},
   issn={0030-8730},
}

\bib{OO13}{article}{
   author={Odaka, Yuji},
   author={Okada, Takuzo},
   title={Birational superrigidity and slope stability of Fano manifolds},
   journal={Math. Z.},
   volume={275},
   date={2013},
   number={3-4},
   pages={1109--1119},
}

\bib{OS12}{article}{
   author={Odaka, Yuji},
   author={Sano, Yuji},
   title={Alpha invariant and K-stability of $\Bbb Q$-Fano varieties},
   journal={Adv. Math.},
   volume={229},
   date={2012},
   number={5},
   pages={2818--2834},
}

\bib{Pie79}{article}{
   author={Piene, Ragni},
   title={Ideals associated to a desingularization},
   conference={
      title={Algebraic geometry (Proc. Summer Meeting, Univ. Copenhagen,
      Copenhagen, 1978)},
   },
   book={
      series={Lecture Notes in Math.},
      volume={732},
      publisher={Springer, Berlin},
   },
   date={1979},
   pages={503--517},
}

\bib{Puk87}{article}{
   author={Pukhlikov, A. V.},
   title={Birational isomorphisms of four-dimensional quintics},
   journal={Invent. Math.},
   volume={87},
   date={1987},
   number={2},
   pages={303--329},
}

\bib{Puk88}{article}{
   author={Pukhlikov, A. V.},
   title={Birational automorphisms of a three-dimensional quartic with a
   simple singularity},
   language={Russian},
   journal={Mat. Sb. (N.S.)},
   volume={135(177)},
   date={1988},
   number={4},
   pages={472--496, 559},
   translation={
      journal={Math. USSR-Sb.},
      volume={63},
      date={1989},
      number={2},
      pages={457--482},
   },
}

\bib{Puk98}{article}{
   author={Pukhlikov, A. V.},
   title={Birational automorphisms of Fano hypersurfaces},
   journal={Invent. Math.},
   volume={134},
   date={1998},
   number={2},
   pages={401--426},
}

\bib{Puk02}{article}{
   author={Pukhlikov, A. V.},
   title={Birationally rigid Fano hypersurfaces},
   language={Russian, with Russian summary},
   journal={Izv. Ross. Akad. Nauk Ser. Mat.},
   volume={66},
   date={2002},
   number={6},
   pages={159--186},
   translation={
      journal={Izv. Math.},
      volume={66},
      date={2002},
      number={6},
      pages={1243--1269},
   },
}

\bib{Puk02b}{article}{
   author={Pukhlikov, A. V.},
   title={Birationally rigid Fano hypersurfaces with isolated singularities},
   language={Russian, with Russian summary},
   journal={Mat. Sb.},
   volume={193},
   date={2002},
   number={3},
   pages={135--160},
   translation={
      journal={Sb. Math.},
      volume={193},
      date={2002},
      number={3-4},
      pages={445--471},
   },
}


\bib{Puk05}{article}{
   author={Pukhlikov, A. V.},
   title={Birational geometry of Fano direct products},
   language={Russian, with Russian summary},
   journal={Izv. Ross. Akad. Nauk Ser. Mat.},
   volume={69},
   date={2005},
   number={6},
   pages={153--186},
   translation={
      journal={Izv. Math.},
      volume={69},
      date={2005},
      number={6},
      pages={1225--1255},
   },
}

\bib{Puk21}{article}{
   author={Pukhlikov, A. V.},
   title={Birational geometry of varieties, fibred into complete intersections of codimension two},
   note={Preprint, {\tt arXiv:2101.10830}},
   date={2021},
}

\bib{Puk15}{article}{
   author={Pukhlikov, A. V.},
   title={Birationally rigid Fano fibrations. II},
   language={Russian, with Russian summary},
   journal={Izv. Ross. Akad. Nauk Ser. Mat.},
   volume={79},
   date={2015},
   number={4},
   pages={175--204},
   translation={
      journal={Izv. Math.},
      volume={79},
      date={2015},
      number={4},
      pages={809--837},
   },
}

\bib{Seg51}{article}{
    AUTHOR = {Segre, Beniamino},
     TITLE = {On the rational solutions of homogeneous cubic equations in
              four variables},
   JOURNAL = {Math. Notae},
    VOLUME = {11},
      YEAR = {1951},
     PAGES = {1--68},
}

\bib{Sho92}{article}{
   author={Shokurov, V. V.},
   title={Three-dimensional log perestroikas},
   language={Russian},
   journal={Izv. Ross. Akad. Nauk Ser. Mat.},
   volume={56},
   date={1992},
   number={1},
   pages={105--203},
   translation={
      journal={Russian Acad. Sci. Izv. Math.},
      volume={40},
      date={1993},
      number={1},
      pages={95--202},
   },
}

\bib{SZ19}{article}{
   author={Stibitz, Charlie},
   author={Zhuang, Ziquan},
   title={K-stability of birationally superrigid Fano varieties},
   journal={Compos. Math.},
   volume={155},
   date={2019},
   number={9},
   pages={1845--1852},
}

\bib{Tia87}{article}{
   author={Tian, Gang},
   title={On K\"{a}hler-Einstein metrics on certain K\"{a}hler manifolds with
   $C_1(M)>0$},
   journal={Invent. Math.},
   volume={89},
   date={1987},
   number={2},
   pages={225--246},
}

\bib{Xu20}{article}{
   author={Xu, Chenyang},
   title={K-stability of Fano varieties: an algebro-geometric approach},
   note={Preprint, {\tt arXiv:2011.10477}},
   date={2020},
}

		
\bib{Zhu20}{article}{
   author={Zhuang, Ziquan},
   title={Birational superrigidity and $K$-stability of Fano complete
   intersections of index 1},
   note={With an appendix by Zhuang and Charlie Stibitz},
   journal={Duke Math. J.},
   volume={169},
   date={2020},
   number={12},
   pages={2205--2229},
}



\end{biblist}
\end{bibdiv}
\end{document}